\documentclass[aihp,preprint]{imsart}

\RequirePackage[T1]{fontenc}
\RequirePackage{amsthm,amsmath,amstext}
\RequirePackage[numbers]{natbib}
\usepackage[french,english]{babel}

\arxiv{math.PR:1007.2177}

\newtheorem{theorem}{Theorem}
\newtheorem{lemma}[theorem]{Lemma}
\newtheorem{prop}[theorem]{Proposition}

\newtheorem{corollary}[theorem]{Corollary}
\newtheorem{example}{Example}
\newenvironment{remark}{\textit{Remark.}}{}

\def\essinf{\mathop{{\rm ess\>inf}}\limits}
\def\esssup{\mathop{{\rm ess\>sup}}\limits}
\def\udim{\overline{\dim}_B}
\def\ldim{\underline{\dim}_B}
\def\Cl{{\rm Cl}}
\def\summ{\sum\limits}
\def\dist{{\rm dist}}
\def\diam{{\rm diam}}
\def\ulim{\mathop{\overline{\lim}}\limits}
\def\llim{\mathop{\underline{\lim}}\limits}
\def\summ{\mathop{\sum}\limits}

\def\R{{\rm I\!R}}
\def\Rd{{\rm I\!R^d}}
\def\1{{\bf 1}}
\def\card{{\rm card}}
\def\N{{\rm I\!N}}
\def\e{\varepsilon}
\def\w{\omega}
\def\ov{\overline}
\def\un{\underline}
\def\Int{{\rm Int}}
\def\D{\Delta}

\begin{document}
\tolerance 5000
\begin{frontmatter}

\title{On random fractals with infinite branching: definition, measurability, dimensions.}

\runtitle{On random fractals with infinite branching: definition, measurability, dimensions.}

\begin{aug}
\author{\fnms{Artemi} \snm{Berlinkov}\thanksref{t11}\ead[label=e1]{artemiberlinkov@hotmail.com}}

\thankstext{t11}{Research was partially supported by NSF grant DMS-0100078 and
DFG-Graduiertenkolleg ''Approximation und algorithmische Verfahren'' in Jena.}

\runauthor{Artemi Berlinkov}

\address{Komendantski pr. 7-1-25, St. Petersburg, 197227, Russian Federation\\
\printead{e1}}

\end{aug}

\selectlanguage{english}
\begin{abstract}
We discuss the definition and measurability questions of random fractals with
infinite branching, and find, under certain conditions,
a formula for the upper and lower Minkowski dimensions.
For the case of a random self-similar set we obtain the packing dimension.
\end{abstract}

\selectlanguage{french}
\begin{abstract}[language=french]
Nous discutons les questions de d\'efinition et de la mesurabilit\'e des fractales
al\'eatoires avec ramification infinie, et trouvons sous certaines conditions
une formule pour les dimensions de Minkowski sup\'erieure et inf\'erieure.
En cas d'ensemble al\'eatoire auto-similaire nous obtenons la dimension d'entassement.
\end{abstract}

\begin{keyword}[class=AMS]
\kwd[Primary ]{28A80}
\kwd[; secondary ]{28A78}
\kwd{60D05}
\kwd{37F40}
\end{keyword}

\begin{keyword}
\kwd{packing dimension}
\kwd{Minkowski dimension}
\kwd{random fractal}
\end{keyword}



\end{frontmatter}

\selectlanguage{english}

\begin{section}{Introduction}
In this paper we study the Minkowski and packing dimensions of random fractals
with infinite branching.

The almost sure Hausdorff dimension of random fractals
was independently found by Mauldin and Williams in \cite{MW}, and Falconer in \cite{F}.
Packing dimension and measures in case of finite branching
were investigated by Berlinkov and Mauldin in \cite{BM}. It was shown
that if the number of offspring is uniformly bounded,
the Hausdorff, packing, lower and upper Minkowski
dimensions coincide a.s.

Barnsley et al. in \cite{BHS1} introduced the notion of $V$-variable fractals and in
\cite{BHS2} find their Hausdorff dimension. Fraser in \cite{Fr}
discusses the Minkowski dimension,
packing and Hausdorff measures from topological (in the Baire sense) point 
of view rather than probabilistic. Random fractals find interesting
and important applications in other areas, e.g. harmonic analysis (\cite{BS2}),
stochastic processes and random fields (\cite{SX}).
 
However most authors focus on the situation when the fractals are finitely branching,
or, in other words, the number of offspring is bounded.
In this paper we investigate the case when the number of offspring may be infinite.
If it is bounded but not uniformly, the results 
of this paper show that all of these dimensions still coincide.
If the number of offspring is unbounded, these dimensions may differ from
each other, as shown in examples \ref{e1}, \ref{e2} from section \ref{examples}.
As we see in these examples, the Minkowski dimensions may
be non-degenerate random variables, whereas in \cite{BM} for the case of finite branching
they have been shown to coincide with the a. s. constant Hausdorff dimension.

In section \ref{sectiondef} we give a precise definition of a random recursive construction
and show that another definition used in \cite{BHS1} for random fractals coincides with it. 
Since in the case of infinite branching the Minkowski dimensions no longer
have to be constant, their measurability is proven in section \ref{measurability}.
In section \ref{results} we derive the Minkowski dimensions of
random recursive constructions under some additional
conditions and a formula for the packing dimensions of
random self-similar sets with infinite branching.
\end{section}

\begin{section}{On the definition of random fractals.}\label{sectiondef}
Let $n\in\N\cup\{\infty\},$ $\D=\{1,\dots,n\}$ if $n<\infty,$
and $\D=\N$ if $n=\infty$.
Denote by $\D^*=\mathop{\bigcup}\limits_{j=0}^\infty \D^j$
the set of all finite sequences of numbers in $\D$,
and by $\D^\N$ the set of all their infinite sequences.
The result of concatenation of two
finite sequences $\sigma$ and $\tau$ from $\D^*$ is denoted
by $\sigma*\tau.$ For a finite sequence $\sigma$, its length
will be denoted by $|\sigma|.$
For a sequence $\sigma$ of length at least $k,$
$\sigma|_k$ is a sequence consisting
of the first $k$ numbers in $\sigma.$ There is a natural
partial order on the $n$-ary tree $\D^*:$
$\sigma\prec\tau$ if and only if the sequence $\tau$
starts with $\sigma.$ A set $S\subset\D^*$ is called an antichain,
if $\sigma \not \prec\tau$ and $\tau \not \prec\sigma$ for all
$\sigma,\tau\in\D^*$.

The following construction was proposed by Mauldin and
Williams in \cite{MW}. We have to modify the original definition
to fully take into account the case of offspring degeneration
(see condition \eqref{cond5}) below.

Suppose that $J$ is a compact subset of ${\R}^d$ such that
$J=\Cl(\Int(J)),$ without loss of generality its diameter equals one.
The construction is a probability space $(\Omega, \Sigma, P)$
with a collection of random subsets of 
${\R}^d$ -- $\{J_\sigma(\w)|\w\in\Omega, \sigma\in\D^*\}$,
so that the following conditions hold.

\renewcommand{\theenumi}{\roman{enumi}}
\begin{enumerate}
\item\label{cond0} $J_\emptyset(\w)=J$ for almost all $\w\in\Omega,$
\item\label{cond1} For all $\sigma\in\D^*$ the maps $\w\to J_\sigma(\w)$
are measurable with respect to $\Sigma$ and the topology generated by the Hausdorff
metric on the space of compact subsets,
\item\label{cond2} For all $\sigma\in\D^*$ and $\w\in\Omega$,
the sets $J_\sigma$, if non-empty, are geometrically similar to $J$
\footnote{The sets $A , B\subset\Rd$ are geometrically similar, if there exist $S:\Rd\to\Rd$ and $r>0$
such that for all $x,y\in\Rd$ $\dist(S(x),S(y))=r\,\dist(x,y)$ and $S(A)=B,$ such $S$
is called a similarity map.},
\item\label{cond3} For almost every $\w\in\Omega$ and all $\sigma\in\N^*,$ 
$i\in\N,$ $J_{\sigma*i}$ is a proper subset of $J_\sigma$ provided  $J_\sigma \neq \emptyset,$
\item\label{cond4} The construction satisfies the random {\it open set condition}:
if $\sigma$ and $\tau$ are two distinct sequences of the same length, then 
$\Int(J_\sigma)\cap\Int(J_\tau)=\emptyset$ a.s. and, finally,
\item\label{cond5}
The random vectors ${\mathbf T}_\sigma=(T_{\sigma*1}, T_{\sigma*2}, \dots),$
$\sigma\in\N^*$, are conditionally i.i.d. given that $J_\sigma(\w)\ne\emptyset$,
where $T_{\sigma*i}(\w)$ equals the ratio of the diameter of $J_{\sigma*i}(\w)$ to the diameter
of $J_\sigma(\w)$ .
\end{enumerate}

The object of study is the random set 
$$K(\w)=\mathop{\bigcap}\limits_{k=1}^\infty
\mathop{\bigcup}\limits_{\sigma\in\D^k}J_\sigma (\w).$$

In general in condition \eqref{cond2} other classes of functions
instead of similarities may be used, e.g. conformal or affine mappings.

The meaning of condition \eqref{cond5} is the following.
Given that $J_\sigma$ is non-empty, we ask that the 
random vectors of reduction ratios 
${\mathbf T}_\sigma=(T_{\sigma*1},T_{\sigma*2},\dots),$ 
have the same conditional distribution and be conditionally independent,
i.e. for any finite antichain $S\subset\D^*$ and any
collection of Borel sets $B_s\subset [0,1]^\D,$ $s\in S,$
\[ P({\mathbf T}_s\in B_s\ \forall s\in S|J_s\ne\emptyset\ \forall s\in S)=
\prod_{s\in S}P({\mathbf T}_s\in B_s|J_s\ne\emptyset),\]
and $\bf T_\sigma$ has the same distribution as $\bf T_\emptyset,$ provided 
$J_\sigma\ne\emptyset,$ i.e. for any $\sigma\in\D^*$ and any Borel set $B\subset\R^\D,$
\[ 
P(\bf{T}_\sigma\in B|J_\sigma\ne\emptyset)
=P(\bf{T}_\emptyset\in B).
\]
Following \cite{MW} the above is called a random recursive construction.

The second term commonly used is
``random fractals'' (see, e.g. \cite{BHS1}),
where condition \eqref{cond5} is replaced by existence
of an i.i.d. sequence of random vectors, such that the equality mentioned
in that condition holds. We note that the following holds:

\begin{prop}
 Random fractals and random recursive constructions are the same class of sets.
\end{prop}
\begin{proof}
That every random recursive construction is a random fractal is obvious
because we can set the distributions of ${\bf T_\sigma}$ given ${J=\emptyset}$
the same as ${\bf T_\emptyset}$.

Suppose that we have a random fractal. Then the random vector ${\bf T_\sigma}$
is independent of vectors ${\bf T_\tau}$ with $\tau\prec\sigma$ and, in particular,
of the event $J_\sigma\ne\emptyset$, therefore the second equality for the random 
vectors being conditionally i.i.d. holds. In the first equality the right hand side equals
\[
\prod\limits_{s\in S}P({\mathbf T}_s\in B_s) 
\]
because $S$ is an antichain and ${\mathbf T}_s$ do not depend on
events $\{J_s\ne\emptyset\}$, $s\in S$,
while the left hand side equals the same expression for the same reason.
\end{proof}

Another definition in \cite{MW} for random stochastically geometrically
self-similar sets made no reference to independence in the construction
but a similar kind of conditional independence condition is needed to find the
dimension of the limit set. We call such sets {\bf random self-similar sets}, and
for them not only the reduction ratios but also the maps (see section \ref{results})
that map parent to its offspring are conditionally i.i.d.
\end{section}

\begin{section}{Preliminaries.}
If the average number of offspring does not exceed one,
then $K(\w)$ is almost surely an empty set or a point, and we
exclude that case from further consideration.
Mauldin and Williams in \cite{MW} have found the Hausdorff
dimension of almost every non-empty set $K(\w),$ 
\[
\alpha=\inf\bigg\{\beta|E\Big[\summ_{i=1}^n T_i^\beta\Big]\le 1\bigg\}.
\]
In case $n<\infty,$ $\alpha$ is the solution of equation
\[
E\Big[\summ_{i=1}^n T_i^\alpha\Big]=1.
\]

The definitions and properties of Hausdorff and packing measures and
dimensions, as well as definitions of upper and lower Minkowski dimensions,
can be found in the book of Mattila (\cite{M}).
We denote the Hausdorff, packing, lower and upper Minkowski
dimension by $\dim_H,$ $\dim_P,$ $\un{\dim}_B$ and 
$\ov{\dim}_B$ respectively. 

For any $K\subset\R^d$ denote by $N_r(K)$ the smallest number of closed
balls with radii $r$, needed to cover $K.$ Then the upper
Minkowski dimension, $$\udim K=\ulim_{r\to 0}-N_r(K)/\log r,$$
and the lower Minkowski dimension,
$$\ldim K=\llim_{r\to 0}-N_r(K)/\log r.$$
Denote by $\ov{M}$ the closure of a set $M$. Obviously, if $M$ is bounded, then
$\udim M=\udim\ov{M}$ and $\ldim M=\ldim\ov{M}$ (see, e.g., \cite{F1}, Proposition~3.4).
One can use the maximal number of disjoint balls of radii $r$
with centers in $K$ (which will be denoted by $P_r(K)$)
instead of the minimal number of balls needed to cover set $K$ 
in the definition of Minkowski dimensions because of the following relation
(\cite{F1}, (3.9) and (3.10)): 
\[
N_{2r}(K)\le P_r(K)\le N_{r/2}(K).
\]
The packing dimension can be defined using upper Minkowski
dimension:
\[
\dim_P K=\inf\{\sup\udim F_i\vert K\subset\cup_i F_i\}. 
\]
\end{section}

\begin{section}{Measurability of Minkowski dimensions.}\label{measurability}

The measurability questions of dimension functions in deterministic case have been studied
by Mattila and Mauldin in \cite{MM}.
We start by exploring these questions for random fractals. In case of finite branching
there is an obvious topology with respect to which the functions $\w\mapsto \udim K(\w)$ and
$\w\mapsto \ldim K(\w)$ are measurable -- the topology generated on the space of compact subsets
of $J$ by the Hausdorff metric. However, it is unknown to the author, with repect to which topology
these maps would be measurable in the case of infinite branching. Therefore we circumvent this
problem as follows.

Denote by ${\mathcal K}(J)$ the space of compact subsets of $J$ equipped with the Hausdorff metric
$$d_H(L_1,L_2)=\max\{\sup\limits_{x\in L_1}\dist(x,L_2),
\sup\limits_{y\in L_2}\dist(L_1,y)\}.$$

\begin{lemma}\label{setsleveln}
Suppose that $L_i\in\mathcal{K}(J)$, $i\in\N$. Then
\[\lim\limits_{k\to+\infty}\mathop{\bigcup}\limits_{i=1}^k L_i=\ov{\mathop{\bigcup}\limits_{i=1}^{+\infty} L_i}\text{ 
in the Hausdorff metric. }\]
\end{lemma}
\begin{proof}
Suppose that \[\lim\limits_{n\to+\infty}d_H\bigg(\mathop{\bigcup}\limits_{i=1}^n L_i,
\ov{\mathop{\bigcup}\limits_{i=1}^{+\infty} L_i}\bigg)>0.\]
Since $\mathop{\bigcup}\limits_{i=1}^n L_i\subset \ov{\mathop{\bigcup}\limits_{i=1}^{+\infty} L_i}$,
there exists an $\e>0$ such that for every $n\in\N$ there exists
$p_n\in \ov{\mathop{\bigcup}\limits_{i=1}^{+\infty} L_i}$
with $\dist(p_n, \mathop{\bigcup}\limits_{i=1}^n L_i)\ge\e$.
Without loss of generality we can assume that $p_n$ converges
to some $p\in\ov{\mathop{\bigcup}\limits_{i=1}^{+\infty} L_i}$.
Then $\dist(p, \mathop{\bigcup}\limits_{i=1}^{+\infty} L_i)\ge\e/2$ 
which is a contradiction. 
\end{proof}
\begin{corollary}
 The map $\w\mapsto 
\ov{\mathop{\bigcup}\limits_{\genfrac{}{}{0pt}{}{|\tau|=n}{J_\tau\cap K\ne\emptyset}}
J_\tau(\w)}$ is measurable.
\end{corollary}
\begin{corollary}
If $\tau_i$, $i\in\N$, is an enumeration of
$\{\tau\in\D^n\vert J_\tau\cap K\ne\emptyset\}$, then 
\[\lim\limits_{k\to\infty}
P_r\left(\mathop{\bigcup}\limits_{i=1}^k J_{\tau_i}\right)=
P_r\left(\ov{\mathop{\bigcup}\limits_{\genfrac{}{}{0pt}{}{|\tau|=n}{J_\tau\cap K\ne\emptyset}} J_\tau}\right).\]
\end{corollary}
\begin{proof}
 The statement follows from the fact that
the function $P_r\colon\mathcal{K}(J)\to \R$ is lower semicontinuous
(see, \cite{MM}, remark after Lemma~3.1).
\end{proof}
\begin{lemma}\label{levelnlimit}
In the Hausdorff metric,
$\lim\limits_{n\to\infty}\ov{\mathop{\bigcup}\limits_{\genfrac{}{}{0pt}{}{|\tau|=n}{J_\tau\cap K\ne\emptyset}}
J_\tau(\w)}=\ov{K(\w)}$ for a. e. $\w\in\Omega$.
\end{lemma}
\begin{proof}
According to \cite{MW}, (1.14), $\lim\limits_{n\to\infty}\sup\limits_{\tau\in\D^n}l_\tau=0$ for a.e. $\w\in\Omega$.
Consider such an $\w$. Suppose that 
\[\lim\limits_{n\to\infty}
d_H\left(\ov{\mathop{\bigcup}\limits_{\genfrac{}{}{0pt}{}{|\tau|=n}{J_\tau\cap K\ne\emptyset}}
J_\tau(\w)},\ov{K(\w)}\right)>0
,\]
then there exists an $\e>0$ such that for every $n\in\N$  there exists $p_n\in
\ov{\mathop{\bigcup}\limits_{\genfrac{}{}{0pt}{}{|\tau|=n}{J_\tau\cap K\ne\emptyset}}J_\tau(\w)}$
with $\dist(p_n,\ov{K(w)})\ge\e$. Choose $n_0\in\N$ such that for all $\tau\in\D^*$ of length at least $n_0$ the following holds:
\[l_\tau(\w)<\e/4.\]
Without loss of generality $p_n$ converges to some $p\in J$. Thus $\dist(p,K(w))\ge\e$.
Next choose $n_1\in\N$, $n_1\ge n_0$ such that for all $n\ge n_1$
\[\dist(p_n,p)<\e/4.\]
Since a $3\e/4$ neighborhood of $p_n$ contains a point of $K(\w)$, we get a contradiction.
\end{proof}
\begin{corollary}
$\lim\limits_{n\to+\infty}
N_r\left(\mathop{\bigcup}\limits_{\genfrac{}{}{0pt}{}{|\tau|=n}{J_\tau\cap K\ne\emptyset}}J_\tau(\w)\right)=N_r(K(\w))$ for a.e. $\w$. The equality holds
if either set is replaced with its closure.
\end{corollary}
\begin{proof}
This follows from the facts that
the function $N_r\colon\mathcal{K}(J)\to \R$ is upper semicontinuous
(see, e.g., \cite{MM}, proof of Lemma~3.1) and $N_r(A)=N_r(\ov{A})$.
\end{proof}

From the statements above follows
\begin{theorem}\label{meas} 
The maps $\w\to\udim K(\w)$ and $\w\to \ldim K(\w)$ are measurable.
\end{theorem}
\begin{proof}
Since the maps
\[
\w\to \ov{K(\w)},\ 
\w\to N_r(\ov{K(\w)})\text{ and } 
\w\to N_r(K(\w))
\]
are measurable, the measurability of the lower and
upper Minkowski dimensions of $K(\w)$ follows from their definition.
\end{proof}

\end{section}

\begin{section}{Dimensions of random fractals.}\label{results}

In this section we derive several expressions for Minkowski and packing dimensions of random self-similar fractals
with infinite branching.

\begin{lemma}
Suppose that $t>\dim_H K$ a.s., $0<p=E\bigg[\summ_{i\in\D}T_i^t\bigg]<1$ and $q\in\N$.
If $\Gamma$ is an arbitrary (random) antichain such that $|\tau|\ge q$
for all $\tau\in\Gamma$ a.s.,
then $E\bigg[\summ_{\tau\in\Gamma}l_\tau^t\bigg]\le \frac{p^q}{1-p}$.
\end{lemma}
\begin{proof}
Indeed,
$E\bigg[\summ_{\tau\in\Gamma}l_\tau^t\bigg]\le\summ_{k=q}^{+\infty}E\bigg[\summ_{|\tau|=k}l_\tau^t\bigg]\le\summ_{k=q}^{+\infty}p^k=\frac{p^q}{1-p}$.
\end{proof}

We will also need the following 2 conditions:
\renewcommand{\theenumi}{\roman{enumi}}
\begin{enumerate}
\setcounter{enumi}{6}
\item the construction is pointwise finite, i.e.
each element of $J$ belongs a.s. to at most finitely many
sets $J_i$, $i\in\N$ (see \cite{MU1}) and
\item\label{nbp} $J$ possesses the {\it neighborhood
boundedness property} (see \cite{GMW}):
there exists an $n_0\in\N$ such that for every $\e>\diam(J),$
if $J_1,\dots,J_k$ are non-overlapping sets which are all similar to $J$
with $\diam(J_i)\ge\e>\dist(J,J_i); i=1,\dots,k,$ then $k\le n_0.$
\end{enumerate}

As we will see, knowledge of similarity maps is essential to find the Minkowski dimension.
For $\tau\in\D^*,$ let $K_\tau(\w)=\mathop{\bigcup}\limits_{\substack{\eta\in\D^\N\\\eta|_{|\tau|}=\tau}}
\mathop{\bigcap}\limits_{i=1}^\infty J_{\eta|_i}(\w)\subset J_\tau(\w)\cap K(\w).$
Fix a point $a\in\R^d$ with $\dist(a,J)\ge 1.$
Denote by $S^\tau_\sigma:\R^d\to\R^d$ a random similarity map such that
$S^\tau_\sigma(J_\tau)=J_{\tau*\sigma}.$ If $J_\tau=\emptyset$ or
$J_{\tau*\sigma}=\emptyset$, then we let
$S^\tau_\sigma(\R^d)=a.$ For a finite word $\sigma\in\N^*,$
let $l_\sigma=\diam(J_\sigma).$
From \cite{MW} we know that
$\lim\limits_{k\to\infty}\sup\limits_{|\tau|=k}l_\tau=0$ a.s.
For $x\in J_\tau$ and $n\in\N,$ consider the random $n$-orbit of $x$
within $J_\tau,$ $O_\tau(x,n)=\mathop{\bigcup}\limits_{\substack{|\sigma|=n\\
J_{\tau*\sigma}\cap K\ne\emptyset}}
S^\tau_\sigma(x).$ For $I\subset\N^*$,
let $O_\tau(x,I)=\mathop{\bigcup}\limits_{\substack{\sigma\in I\\
J_{\tau*\sigma}\cap K\ne\emptyset}}S^\tau_\sigma(x).$
In case $\tau=\emptyset,$ $O_\tau(x,I)$ is denoted by $O(x,I),$
$O_\tau(x,n)$ by $O(x,n)$, and $S^\tau_\sigma$ by $S_\sigma.$

\noindent {\it Acknowledgement. } That in the following lemma
(analogous to Proposition 2.9 in \cite{MU2})
condition \eqref{nbp} is sufficient became known to the author during conversation
with R.~D.~Mauldin.

\begin{lemma}\label{lemmamink}
For all $\w\in\Omega,$ $n\in\N,$ and any two collections of points
$X=\{x_k\}_{k=1}^\infty,$ $Y=\{y_k\}_{k=1}^\infty\subset\mathop{\cup}
\limits_{|\sigma|=n}J_\sigma$ such that for
all $\sigma\in\D^n$ 
$\card(Y\cap J_\sigma)=\card(X\cap J_\sigma)=1\text{ or } 0,$
$\udim X=\udim Y$ and $\ldim X=\ldim Y.$
\end{lemma}

\begin{proof}
Without loss of generality we assume that $n=1$
since for every $n>1$ the collection of sets $\{J_\tau\}$
such that $|\tau|$ is divisible by $n$ forms a random recursive
construction. First we note that there exists an $M>0$ such that
$$\forall r>0\ \forall z\in\R^d\ \card\{i\in\N|B(z,r)\cap J_i(\w)\ne\emptyset
{\rm\ and\ }l_i(\w)\ge r/2\}\le M.$$
Fix $\w\in\Omega,$ $z\in\R^d$, $r>0.$ Obviously $B(z,r)$ can be covered by
$12^d$ balls of radius $r/6$. Let $B_1$ be one of them and place
inside $B_1$ a set similar to $J.$  By the neighborhood boundedness
property with $\e=r/2$, we obtain
$\card\{i\in\N|B_1\cap J_i\ne\emptyset{\rm\ and\ }l_i\ge r/2\}\le n_0.$
Therefore it suffices to take $M=12^dn_0.$

Finally take $0<r\le 2,$ let $I_r(\w)=\bigcup\limits_{l_i(\w)< r/2}J_i(\w)$
and $I^\prime_r(\w)=\bigcup\limits_{l_i(\w)\ge r/2} J_i(\w).$
Then $N_r(Y\cap I_r)\le N_{r/2}(X\cap I_r).$ Clearly,
for any collection of points $Z=\{z_k\}_{k=1}^\infty,$
such that $\card (Z\cap J_i)=0\text { or }1$ for all $i,$
we have $N_r(Z\cap I_r^\prime)\le\card(I_r^\prime)$.
On the other hand $N_r(Z\cap I_r^\prime)\ge\card(I_r^\prime)/M.$
Hence, $$N_r(Y)\le N_{r/2}(X\cap I_r)+N_r(Y\cap I_r^\prime)
\le N_{r/2}(X)+MN_r(X\cap I_r^\prime)
\le (1+M)N_{r/2}(X).$$
The result follows.
\end{proof}

\begin{remark} From the proof of lemma~\ref{lemmamink}, we see that if
for some $x\in J,$ $D>0$ and $0\le u\le d$ for all $0<r\le 2,$
$N_r(O(x,1))\le Dr^{-u}$, then for all $y\in J,$
$N_r(O(y,1))\le 2^d(12^d n_0+1)Dr^{-u}.$
\end{remark}

For $\tau\in\N^*$, let $\ov{\gamma}_\tau=\udim O_\tau(x,1)$ for some
$x\in J_\tau$ and let $\ov{\gamma}=\sup\limits_{\tau\in\D^*}\ov{\gamma}_\tau.$
By lemma~\ref{lemmamink}, $\ov{\gamma}_\tau$ does not depend on the choice of $x.$
Similarly we define $\un{\gamma}_\tau=\ldim O_\tau(x,1)$ and 
$\un{\gamma}=\sup\limits_{\tau\in\D^*}\un{\gamma}_\tau.$
For the rest of the paper, suppose additionally that 
\begin{enumerate}
\setcounter{enumi}{8}
\item\label{regorbits}
there exists $A>0$ such that for all $\tau\in\D^*,$
$x\in J_\tau,$ $t>0$ and $0<r\le 2$ we have
$N_{r}(O_\tau(x,1)) \1_{\{\ov{\gamma}_\tau<t\}}\le Ar^{-t}l_\tau^t.$
\end{enumerate}
\renewcommand{\theenumi}{\arabic{enumi}}

\begin{lemma}\label{lemmaorbits}
For any $x\in J,$ 
\[
\max\{\dim_H K,\sup\limits_n\udim O(x,n)\}
=\max\{\dim_H K,\ov{\gamma}\}\text{ and }
\]
\[
\max\{\dim_H K,\sup\limits_n\ldim O(x,n)\}
=\max\{\dim_H K,\un{\gamma}\}\text{ a.s.}
\]
\end{lemma}

\begin{proof} 
Fix $\w\in\Omega.$ Since for any $\tau\in\N^*,$
$O_\tau(S_\tau(x),1)\subset O(x,|\tau|+1),$ we have
$\ov{\gamma}_\tau=\udim O_\tau(S_\tau(x),1)\le\udim O(x,|\tau|+1)\le
\sup\limits_n\udim O(x,n),$ and
$\ov{\gamma}\le\sup\limits_n\udim O(x,n).$

In the opposite direction we prove by induction
on $n$ that if $P(\max\{\dim_H K,\ov{\gamma}\}<t)>0$ for some $t>0$, then
there exists a random variable $B_n>0$ such that $E[B_n]<+\infty$ and
$N_r(O(x,n))\1_{\{\ov{\gamma}<t\}}\le B_n r^{-t}$ a.s.
for all $0<r\le 1$.
When $n=1,$ we let $B_1=A.$
Suppose that for all $n\le k$ and for all $0<r\le 1,$ there exists
$B_n>0$ with $E[B_n]<+\infty$ such that
$N_r(O(x,n))\1_{\{\ov{\gamma}<t\}}\le B_n r^{-t}$
a.s. To prove the statement for $n=k+1,$ fix $r>0$ and
set $I_r(\w)=\{\tau\in\N^k|l_\tau(\w)<r/2\}.$
Then $$N_r(O(x,I_r\times\N))\le
N_{r/2}(O(x,I_r))\le N_{r/2}(O(x,k)).$$
For a fixed $\tau\in\N^k,$
$$N_r(O_\tau(S_\tau(x),1))\1_{\tau\not\in I_r}
\1_{\{\ov{\gamma}<t\}}\le A l_\tau^t r^{-t}.$$
Therefore $$N_r(O(x,k+1))\1_{\{\ov{\gamma}<t\}}\le 
N_{r/2}(O(x,k))\1_{\{\ov{\gamma}<t\}}+$$
$$+\summ_{|\tau|=k}N_r(O_\tau(S_\tau(x),1))\1_{\tau\not\in I_r}
\1_{\{\ov{\gamma}<t\}}\le 2^t B_kr^{-t}+A r^{-t}
\summ_{|\tau|=k}l_\tau^t.$$
Set $B_{k+1}=2^t B_k+A\summ_{|\tau|=k}l_\tau^t.$
If we fix $n$, then by Markov's inequality for every $\e>0$
$$\summ_{i=0}^\infty P(B_n 2^{it}>2^{i(t+\e)})\le
\summ_{i=0}^\infty E[B_n] 2^{-i\e}<\infty,$$ and therefore by Borel-Cantelli lemma
for a.e. $\w\in\Omega$ $B_n 2^{it}>2^{i(t+\e)}$ only
finitely many times, hence for a.e. $\w\in\Omega$
$N_{2^{-i}}(O(x,n))\1_{\{\ov{\gamma}<t\}}>2^{i(t+\e)}$
only finitely many times.
Therefore $$\varlimsup\limits_{i\to\infty}\frac{\log N_{2^{-i}}(O(x,n))}{i\log 2}<t+\e$$
for almost every $\w$ such that $\max\{\dim_H K(\w),\ov{\gamma}(\w)\}<t$ for every $\e>0$. Thus
for almost every such $\w$ we have $\udim O(x,n)\le t$. 
The same argument holds for the lower Minkowski dimension.
\end{proof}
From the proof of the last lemma and the fact that there cannot be more than $10^d$ offspring in the construction of diameter at least 1/5 follows
\begin{corollary}\label{hangorbits}
 Suppose that $q\in\N$, construction satisfies property \eqref{regorbits},
for some $t>0$ $P(\max\{\dim_H K,\ov{\gamma}\}< t)>0$ and let
$$\Gamma_{\tau,q}=
\{\eta\in\D^{|\tau|+q}\colon l_\eta<l_\tau/5\}\cup
\{\eta\in\D^*\colon|\eta|>|\tau|+q, l_\eta<l_\tau/5,
l_{\eta|_{|\eta|-1}}\ge l_\tau/5\}.$$
Then there exists a random variable $B'_q$ with
$E[B'_q]<+\infty$ such that
$$N_r(O(x,\Gamma_{\tau,q}))\1_{\{\ov{\gamma}<t\}}\le B'_q l_\tau^t r^{-t}.$$
\end{corollary}
\begin{proof}
Let $$\Gamma_{0,\tau,q}=
\{\sigma\in\D^*\colon|\sigma|\ge |\tau|+q, l_\sigma\ge l_\tau/5,
 \exists\tau\in\Gamma_{\tau,q}\colon \tau|_{|\tau|-1}=\sigma\}.$$
Then 
$$N_r(O(x,\Gamma_{\tau,q}))\1_{\{\ov{\gamma}<t\}}\le
N_r(O_\tau(S_\tau(x),q))\1_{\{\ov{\gamma}<t\}}+
\summ_{\sigma\in\Gamma_{0,\tau,q}}N_r(O_\sigma(S_\sigma(x),1))\1_{\{\ov{\gamma}<t\}}$$
$$\le B_q l_\tau^t r^{-t}+Al_\tau^t r^{-t}\card\{\sigma\in\D^*|l_\sigma\ge 1/5\},$$
where $B_q$ and the estimate on the first term come from the
proof of lemma~\ref{lemmaorbits}, and the second term is bounded
according to condition \eqref{regorbits}.

Note that if $0<p=E\left[\summ_{i\in\Delta}T^t_i\right]<1$, then $$E\left[\summ_{|\tau|=q}\frac{l_\tau^t}{(1/5)^t}\right]=
5^tE\left[\summ_{|\tau|=q}l_\tau^t\right]=5^tp^q
\ge E[\card\{\tau|\tau\in\D^q,\ l_\tau\ge 1/5\}].$$
Hence 
$$E[\card\{\sigma\in\D^*|l_\sigma\ge 1/5\}]=\summ_{k=1}^{+\infty}
E[\card\{\sigma\in\D^k|l_\sigma\ge 1/5\}]\le\frac{5^t}{1-p}
$$
and we can put $B'_q=B_q+A\card\{\sigma\in\D^*|l_\sigma\ge 1/5\}.$
\end{proof}

\begin{lemma}\label{lemmamain}
For every $t\in\R$ such that $P(\max\{\dim_H K,\ov{\gamma}\}<t)>0,$
$\udim K\le t$ for a. e. $\w$ such that $\ov{\gamma}(\w)<t$.
\end{lemma}

\begin{proof}

Suppose that $P(\max\{\dim_H K,\ov{\gamma}\}<t)>0.$
Let $p\in(0,1)$ be defined by equality $p=E\left[\summ_{i\in\N}l_i^t\right]$.
We will prove by induction on $n$ that there exists $B>0$
such that for each $n$, for every $\tau\in\D^*$ there exists a random
variable $B_{\tau,n}$, independent of the $\sigma$-algebra generated
by the maps $\w\mapsto l_{\tau|_i}(\w)$, $1\le i\le |\tau|$,
with $E\big[B_{\tau,n}]\le B$ such that
$$N_r(K_\tau)\1_{\{\ov{\gamma}<t\}}\le B_{\tau,n} r^{-t}l_\tau^t\text{ for a.e. }\w
\text{ such that } 1/n\le r/l_\tau(w)\le 1.$$

Choose $q\in\N$ such that $p^q<1/2$. Then put $B=\max\{2^d,4^{t+1}E[B_q']\}$,
where $B'_q$ is the random variable from corollary~\ref{hangorbits}.
The induction base obviously holds for $n=1,2$.

Suppose the statement is true for $n_0\in\N,$ and $1/(n_0+1)\le r<1/n_0.$
We can assume that $K_\tau\ne\emptyset$.
Let $$C_{\tau,1}(\w)=\left\lbrace\sigma\in\Gamma_{\tau,q}\,|
\,l_\sigma\le \frac{l_\tau}{2n_0+2}\right\rbrace,
C_{\tau,2}(\w)=\left\lbrace\sigma\in\Gamma_{\tau,q}\,|
\,l_\sigma> \frac{l_\tau}{2n_0+2}\right\rbrace,$$
where $$\Gamma_{\tau,q}=\{\sigma\in\D^{q+|\tau|}\colon
 l_\sigma<l_\tau/5\}\cup\{\sigma\in\D^*\colon
|\sigma|>q+|\tau|, l_\sigma<l_\tau/5, l_{\sigma|_{|\sigma|-1}}\ge l_\tau/5\}.$$
Since $$K_\tau=\Big(\mathop{\bigcup}\limits_{\sigma\in C_{\tau,1}}K_\sigma\Big)
\cup\Big(\mathop{\bigcup}\limits_{\sigma\in C_{\tau,2}}K_\sigma\Big),$$
we have
$$N_r(K_\tau)\le N_{\frac{1}{n_0+1}}\Big(\mathop{\bigcup}
\limits_{\sigma\in C_{\tau,1}}K_\sigma\Big)+\summ_{\sigma\in C_{\tau,2}}N_r(K_\sigma).$$

We note that $N_{\frac{1}{n_0+1}}\Big(\mathop{\bigcup}\limits_{\sigma\in C_{\tau,1}}
K_\sigma\Big)\le N_{\frac{1}{2n_0+2}}(O(x,\Gamma_{\tau,q}))$ because if
$B(y_j, \frac{1}{2n_0+2})$ is a collection of balls of radius $\frac{1}{2n_0+2}$
covering $O(x,\Gamma_{\tau,q}),$ then the balls $B(y_j,\frac{1}{n_0+1})$ cover
$\mathop{\bigcup}\limits_{\sigma\in C_{\tau,1}}K_\sigma,$ since
$\diam(J_\sigma)<\frac{1}{2n_0+2}$ for all $\sigma\in C_{\tau,1}.$ Therefore
by corollary \ref{hangorbits}
$$N_r(K_\tau)\1_{\{\ov{\gamma}<t\}}\le B'_q l_\tau^t 2^t (n_0+1)^t+
\summ_{\sigma\in\Gamma_{\tau,q}}N_r(K_\sigma)
\1_{\{l_\sigma\in C_{\tau,2}\}}\1_{\{\ov{\gamma}<t\}}
\text{ a.s.}$$

The following chain of inequalities ensures applicability of the induction hypothesis
to estimate the terms in the last sum:
$$\frac{r}{l_\sigma}>5r\ge\frac{5}{n_0+1}>\frac{1}{n_0},$$
therefore
$$N_r(K_\sigma)\1_{\{l_\sigma\in C_{\tau,2}\}}\1_{\{\ov{\gamma}<t\}}
\le B_{\sigma,n}r^{-t} l_\sigma(\w)^t\text{ a.s.}$$
Since $r\le 2/(n_0+1),$
$$N_r(K_\tau)\1_{\{\ov{\gamma}<t\}}\le
r^{-t} \left(4^t l_\tau^t B'_q+\summ_{\sigma\in\Gamma_{\tau,q}}
B_{\sigma,n_0} l_\sigma^t\right)=r^{-t}l_\tau^t
\left(4^t B'_q+\summ_{\sigma\in\Gamma_{\tau,q}}B_{\sigma,n_0}
l_\sigma^t/l_\tau^t\right) \text{ a.s.}$$
Note that
\[
E\left[\left(4^t B'_q+\summ_{\sigma\in\Gamma_{\tau,q}} B_{\sigma,n}
l_\sigma^t/l_\tau^t\right)\right]\le 4^t E[B'_q]+B p^q<B/4+B/2<B.
\]
Applying the same argument as in lemma~\ref{lemmaorbits} we come to the desired conclusion.
\end{proof}

\begin{theorem}\label{mink}
If there exists $A>0$ such that for all
$x\in J_\tau,$ $t>0$ and $0<r\le 2$ we have
$$N_r(O_\tau(x,1))\1_{\{\ov{\gamma}_\tau<t\}}<Ar^{-t}l_\tau^t,$$ then
$\udim K=\max\{\dim_H K,\ov{\gamma}\}$ a.s. provided $K\ne\emptyset.$
Similarly, if
$$N_r(O_\tau(x,1))\1_{\{\un{\gamma}_\tau<t\}}<Ar^{-t}l_\tau^t,$$ then
$\udim K=\max\{\dim_H K,\un{\gamma}\}$ a.s. on $\{K\ne\emptyset\}.$
\end{theorem}
\begin{proof}
Fix $n\in\N$ and consider a collection
of points $X=\{x_i\}_{i=1}^\infty\subset K$ such that for all $\sigma\in\N^n,$
$J_\sigma\cap K\ne\emptyset\Rightarrow\card(X\cap J_\sigma)=1$ and
$J_\sigma\cap K=\emptyset\Rightarrow\card(X\cap J_\sigma)=0.$
By lemma~\ref{lemmamink}, $\udim X=\udim O(x,n)$, and therefore
$\udim K\ge\max\{\dim_H K,\sup\limits_{n\in\N}\udim O(x,n)\}.$
By lemma \ref{lemmamain}, $P(\udim K>\max\{\dim_H K,\ov{\gamma}\})=0.$
\end{proof}

\begin{corollary}\label{finite}
If the number of offspring is finite
almost surely, then $\dim_H K=\dim_P K=\ldim K=\udim K$ a.s.
\end{corollary}

\begin{theorem}\label{packsim}
Suppose that we have a random self-similar set and 
there exists $A>0$ such that $$N_r(O(x,1))<Ar^{-\ov{\gamma}}$$
a.s. for all $0<r\le 2.$ 
Then $\dim_P K=\udim K=\max\{\dim_H K,\esssup\udim O(x,1)\}$ and
$\ldim K=\max\{\dim_H K,\esssup\ldim O(x,1)\}$ a.s. on
$\{K\ne\emptyset\}.$
\end{theorem}

\begin{proof}
Since for a random self-similar set $\ov{\gamma_\tau},$ $\tau\in\N^*$ are 
conditionally i.i.d., we obtain that if $K(\w)\ne\emptyset,$ then
$\ov{\gamma}=\esssup\udim O(x,1)$ a.s.
To see this, let $z=\esssup\udim O(x,1),$
then $\esssup\ov{\gamma_\tau}\le z$ for all $\tau\in\N^*$
and $\ov{\gamma}=\sup\limits_\tau\ov{\gamma_\tau}\le z$ a.s.
If $z=0$ or $\ov{\gamma_\emptyset}=z$ a.s., we are done.
Otherwise consider $0<y<z$ such that 
$$0<P(\udim O(x,1)\le y)=b<1.$$
For all $\tau\in\N^*,$ $b=P(\ov{\gamma_\tau}\le y|J_\tau\ne\emptyset).$
Now we prove that for every $\e\in(0,1)$
\[
P(\{\forall\tau\,\ov{\gamma_\tau}\le y\}\cap \{K\ne\emptyset\})\le\e P(K\ne\emptyset).
\]
Find $m\in\N$ such that $b^m<\e P(K\ne\emptyset)/2.$ From \cite{MW}
it is known that if $S_k$ denotes the number of non-empty offspring on
level $k,$ then for almost every $\w\in\{K\ne\emptyset\},$
$\lim\limits_{k\to\infty}S_k=\infty$,
and for almost every $\w\in\{K=\emptyset\},$  $\lim\limits_{k\to\infty}S_k=0.$
Therefore we can find $\Omega_0\subset\{K\ne\emptyset\},$ $k_0\in\N$,
and perhaps a bigger $m$ such that
\[
 P(\{K\ne\emptyset\}\setminus \Omega_0)<\e P(K\ne\emptyset)/2\text{ and }
\forall \w\in \Omega_0\ S_{k_0}(w)\ge m.
\]
Next we enumerate somehow all indices of $\Delta^{k_0}$ and fix this enumeration,
then denote all $m$-element subsets of $\Delta^{k_0}$ by $F_i$, $i\in\N$.
For $\w\in \Omega_0$ denote the event, that the first $m$ non-empty sets
$J_\sigma(\w),\ \sigma\in\Delta^{k_0}$, concide with $F_i$, by $\Omega_i$.
Then $\Omega_i$ form a partition of $\Omega_0$ and

\[
P(\{\ov{\gamma}\le y\}\cap \{K\ne\emptyset\}))=
P(\{\ov{\gamma}\le y\}\cap \Omega_0)+
P(\{K\ne\emptyset\}\setminus\Omega_0)\le
\]
\[
\le\sum\limits_i P(\{\ov{\gamma}\le y\}\cap \Omega_i)+\e P(K\ne\emptyset)/2=
\sum\limits_i P(\ov{\gamma}\le y|\Omega_i)P(\Omega_i)+\e P(K\ne\emptyset)/2\le
\]
\[
 \le \sum\limits_i b^m P(\Omega_i)+\e P(K\ne\emptyset)/2\le
\e P(K\ne\emptyset).
\]

Examination of the proofs of lemmas \ref{lemmaorbits}, \ref{lemmamain} and theorem \ref{mink} shows that for every
$\tau\in\N^*,$ $\udim K_\tau=\max\{\dim_H K,\esssup O(x,1)\}$
provided $K_\tau\ne\emptyset.$
Now using Baire's category theorem
we see that for
$t<\max\{\dim_H K,\esssup\udim O(x,1)\},$
${\mathcal P}^t(K)=\infty.$
The result follows.
\end{proof}

What is the packing dimension of infinitely branching random fractals
in general is unknown.

\end{section}

\begin{section}{Examples.}\label{examples}
As we see for a random self-similar set the packing dimension is almost surely
constant even with infinite branching. In the following example we see that
if we drop the condition that the similarity maps
are conditionally independent, packing dimension is no longer a constant.

\begin{example}\label{e1}
Random fractal for which the zero-one law does not hold.
\end{example}

Let $J=[0,1]$ and take $p(\w), \w\in\Omega$ with respect to the uniform distribution
on $[1,2].$ We build a random recursive construction so that on level 1,
the right endpoints of the offspring are the points $1/n^p$, $n\in\N,$ and the length
of the $n$-th subinterval is $V_n=(1/16^n)\inf\limits_{1\le p\le 2}\{1/n^p-1/(n+1)^p\}.$
On all other levels, the offspring are formed
from a scaled copy of $[0,1]$ and its disjoint subintervals of length $V_n$
with right endpoints at $1/n^p,$ $n\in\N.$ 
Obviously, $\summ_{n=1}^\infty V_n^{1/4}<\infty,$ and hence
for each $\w\in\Omega,$ we have $\dim_H K\le 1/4.$ On the other
hand we can use the results from \cite{MU2} to determine that
for each $\w\in\Omega,$ $\dim_P K(\w)=\udim K(\w)=\frac{1}{p(\w)+1}.$
So, the reduction ratios are constant, but random placement of subintervals
gives non-trivial variation of the packing dimension.

\begin{example}\label{e2}
Random recursive construction for which $\udim K$ is a non-degenerate
random variable and $\dim_H K< \dim_P K<\essinf\udim K$ a.s.
\end{example}

Note that for $p>0$, $\udim\{1/n^p, n\in\N\}={1/(p+1)}.$
Let $J=[0,1]$ and take $p$ with respect to the uniform distribution
on $[1,2].$ We build a random recursive construction so that on level 1,
the right endpoints of offspring are the points $1/n^p$, $n\in\N.$
On all other levels, the offspring are formed from a scaled copy
of $[0,1]$ and its disjoint subintervals with right endpoints at $1/n^4,$ $n\in\N.$
Let $(V_1, V_2, \dots)$ be a fixed vector of reduction ratios so that
$V_n=(1/1024)^n\inf\limits_{1\le p\le 4}\{1/i^p-1/(i+1)^p\}.$
Then $\summ_{n=1}^\infty V_n^{1/8}<1,$ $K(\w)\ne\emptyset,$ 
$\dim_H K\le 1/8$ and $\udim K=\max\{\dim_H K, 1/(p+1)\}=1/(p+1)$,
where $p$ is chosen according to the uniform distribution on
$[1,2].$ Hence, $\essinf\udim K=1/3.$ By theorem~\ref{packsim},
$\dim_P K=1/5$.
\end{section}


\begin{thebibliography}{9}

\frenchspacing
\bibitem{BHS1} M.~F.~Barnsley, J.~E.~Hutchinson and \"O~Stenflo.
$V$-variable fractals: Fractals with partial self similarity.
{\sl  Adv. Math. \bfseries 218}, No. 6 (2008), pp. 2051--2088.
\bibitem{BHS2} M.~F.~Barnsley, J.~E.~Hutchinson and \"O~Stenflo.
$V$-variable fractals: dimension results.
{\sl Forum Math. \bf 24} (2012), pp. 445--470.
\bibitem{BS2} D.~Beliaev and S.~Smirnov. Random conformal snowflakes.
{\sl Ann. Math. \bf 172}, No. 1 (2010), pp. 597--615.
\bibitem{BM} A.~Berlinkov and R.~D.~Mauldin. Packing measure and dimension
of random fractals. {\sl J. Theoret. Probab. \bf 15}, No. 3 (2002), pp. 695--713.
\bibitem{F} K.~J.~Falconer. Random Fractals.
{\sl Math. Proc. Camb. Philos. Soc. \bf 100} (1986), pp. 559--582.
\bibitem{F1} K.~J.~Falconer. Fractal Geometry. Mathematical Foundations and Applications.
John Wiley \& Sons, Chichester, UK, 2003. 
\bibitem{F2} K.~J.~Falconer. Techniques in Fractal Geometry.
John Wiley \& Sons, Chichester, UK, 1997.
\bibitem{Fr} J.~M.~Fraser. Dimensions and measure for typical random fractals.
{\sl Arxiv preprint arXiv:1112.4541} (2011).
\bibitem{G} S.~Graf. Statistically self-similar fractals.
{\sl Prob. Th. Rel. Fields \bf 74} (1987), pp. 357--392.
\bibitem{GMW} S.~Graf, R.~D.~Mauldin and S.~C.~Williams.
The exact Hausdorff dimension in random recursive constructions.
{\sl Mem. Am. Math. Soc. \bf381}  (1988).
\bibitem{M} P.~Mattila. Geometry of sets and measures in
Euclidean spaces. Cambridge University Press, Cambridge, UK, 1995.
\bibitem{MM} P.~Mattila and R.~D.~Mauldin. Measure and dimension functions:
measurability and densities. {\sl Math. Proc. Camb. Phil. Soc. \bf 121} (1997),
pp. 81--100. 
\bibitem{MU1} R.~D.~Mauldin and  M.~Urbanski. Dimensions
and measures in iterated function systems. {\sl Proc.
London Math. Soc. \bf 73}, No. 3 (1996), pp. 105--154.
\bibitem{MU2} R.~D.~Mauldin and  M.~Urbanski. Conformal iterated
function systems with applications to the geometry of continued
fractions. {\sl Trans. Am. Math. Soc. \bf351} (1999), pp. 4995--5025.
\bibitem{MW} R.~D.~Mauldin and S.~C.~Williams. Random recursive 
constructions: asymtotic geometric and topological properties.
{\sl Trans. Am. Math. Soc. \bf 295} (1986), pp. 325--346.
\bibitem{SX} N.-R.~Shieh and Y.~Xiao. Hausdorff and packing dimensions
of the images of random fields. {\sl Bernoulli, \bf 16}, No. 4 (2010), pp. 926--952.
\end{thebibliography}
\end{document}